\documentclass[reqno0,12pt]{amsart}
\usepackage{dsfont}
\usepackage{bm}
\usepackage{bbm}
\usepackage[nodate]{datetime}
\usepackage[hmargin=30mm,top=28mm,bottom=28mm,a4paper]{geometry}
\usepackage{color}
\usepackage{subfig}
\usepackage{fancyhdr}
\usepackage{amsmath,amssymb,amsfonts,amsthm}
\usepackage{amstext}
\usepackage{amsmath}
\usepackage{amssymb}
\usepackage{enumerate}
\usepackage{amsbsy}
\usepackage{amsopn}
\usepackage{bbm,amsthm}
\usepackage{amscd}
\usepackage{amsxtra}
\usepackage{todonotes}

\newtheorem{theorem}{Theorem}[section]
\newtheorem*{theorem*}{Theorem}
\newtheorem{lemma}{Lemma}[section]

\theoremstyle{remark}

\newcommand{\ind}{\mathds{1}}

\DeclareMathOperator{\vol}{vol}

\begin{document}
\title{$\Omega$-bounds for the partial sums of some modified Dirichlet characters II}
\author{Marco Aymone, Ana Paula Chaves, Maria Eduarda Ramos}
\begin{abstract}
A modified Dirichlet character $f$ is a completely multiplicative function such that for some Dirichlet character $\chi$, $f(p)=\chi(p)$ for all but a finite number of primes $p\in S$, and for those exceptional primes $p\in S$, $|f(p)|\leq 1$. If $\chi$ is primitive and for each $p\in S$ we have $|f(p)|=1$, we prove that $\sum_{n\leq x}f(n)=\Omega((\log x)^{(|S|-3)/2})$. This makes progress on a Conjecture due to Klurman, Mangerel, Pohoata and Ter\"av\"ainen, c.f. Trans. Amer. Math. Soc., 374 (2021), pp. 7967--7990. Our proof combines tools from Analytic Number Theory, Harmonic Analysis, Baker's Theory on linear forms in logarithms and Discrepancy bounds for sequences uniformly distributed modulo $1$. 
\end{abstract}

\maketitle

\section{Introduction.}

In this sequel we continue the study on $\Omega$-bounds\footnote{Given functions $M(x)$ and $N(x)$ such that $N(x)>0$ for all $x$, we recall that $M(x)=\Omega(N(x))$ if $\limsup_{x\to\infty}\frac{|M(x)|}{N(x)}>0$.} for the partial sums of a modified Dirichlet character $f$, \textit{i.e.}, $f$ is completely multiplicative and for some Dirichlet character $\chi$ there is a finite set of primes $S$ such that $f(p)=\chi(p)$ for all primes outside $S$, and for the exceptional primes $p\in S$, $f(p)\neq\chi(p)$ and $|f(p)|\leq 1$. In this case, we say that $f$ is a modification of $\chi$ with modification set $S$.

This class of functions occur naturally in the Erd\H{o}s Discrepancy Problem resolved by Tao \cite{taodiscrepancy}. In this resolution, it was proved that a completely multiplicative function $f$ with $|f(n)|=1$ for all $n$ satisfies
$$\sup_x\left|\sum_{n\leq x}f(n)\right|=\infty.$$

Dirichlet characters play an important role in this problem, since the non-principal ones have bounded partial sums and are completely multiplicative. Therefore they are considered near-counterexamples to the Theorem of Tao, where the word near comes from the fact that they vanish at a finite subset of primes, and hence the condition $|f|=1$ is barely satisfied.

In \cite{klurmanteravainen}, Klurman et al. proved that for a modification $f$ of a primitive Dirichlet character with $|S|\geq 1$ such that, for at least one prime $p\in S$ we have $|f(p)|=1$, then the partial sums  satisfies
$$\sum_{n\leq x}f(n)=\Omega(\log x),$$
and thus improving a result of Borwein, Choi and Coons \cite{Borweindicrepancy}. Further, a corrected version of a Conjecture in \cite{klurmanteravainen} states that if $\chi$ is primitive and $|f(p)|=1$ for each prime $p\in S$, then $\sum_{n\leq x}f(n)=\Omega((\log x)^{|S|})$. We will refer to this as Klurman-Mangerel-Pohoata-Ter\"av\"ainen Conjecture.

In the prequel \cite{aymone_modified} to this present paper it was exhibited examples where the above Conjecture holds. Here we make the following progress on this problem. 

\begin{theorem}\label{Teorema principal} Let $f$ be a modification of a primitive Dirichlet character $\chi$ with modification set $S$, where $|f(p)|=1$ for each $p\in S$. Then
$$\sum_{n\leq x}f(n)=\Omega((\log x)^{(|S|-3)/2}).$$
\end{theorem}

The proof of this result relies in a combination of techniques coming from different branches of Number Theory:
\begin{itemize}
\item Zero free regions, estimates and functional equation for $L(s,\chi)$;
\item Plancherel's identity from Harmonic Analysis;
\item Baker's Theory on linear forms in logarithms;
\item Discrepancy bounds for sequences uniformly distributed modulo $1$.
\end{itemize}

In the following sections we discuss each of these items and how they relate with the proof. Before that we state the current state of an art on the problem of $\Omega$ bounds for the partial sums of modified Dirichlet characters.

Let $T:=|\{p\in S:f(p)=1\}|-|\{p\in S:\chi(p)=1\}|$,
and
$$N:=\begin{cases}\max\{0,T\}, \mbox{ if }\chi(-1)=-1, (\mbox{odd character}),\\\max\{0,T-1\}, \mbox{ if }\chi(-1)=1, (\mbox{even character}).\end{cases}$$
Then Theorem \ref{Teorema principal} combined with the results in \cite{aymone_modified} and \cite{klurmanteravainen} allows us to state the following result.

\begin{theorem} Let $f$ be a modification of a primitive Dirichlet character $\chi$ with modification set $S$, where $|f(p)|=1$ for each $p\in S$. Let $N$ be as above. Then, for $D=\max\{1,N,(|S|-3)/2\}$, we have that
$$\sum_{n\leq x}f(n)=\Omega((\log x)^D).$$
\end{theorem}

Therefore, examples in which the Klurman-Mangerel-Pohoata-Ter\"av\"ainen Conjecture holds can be obtained by modifying an odd primitive Dirichlet character $\chi$ at primes $p$ in some set $S$ where $\chi(p)\neq 1$ and $f(p)=1$, for all primes $p\in S$.

We conclude this introduction by briefly explaining the next sections. In Section \ref{section notation} we state the most frequent notation used in this paper. In Section \ref{section prova} we break in several Subsections the proof of Theorem \ref{Teorema principal} in order to explain each of the items coming from different branches of Number Theory used in the proof.

\section{Notation}\label{section notation} 
Here is a summary of all frequently used notation throughout this paper.  We hope that it may be useful for the reader always when he or she becomes overloaded with the plenty number of notation used here. Some notation less used we prefer to explain right before it is needed.

\subsection{Letters appearing throughout the text} We will let the letter $p$ to always represent a prime number, $n$ to represent a positive integer, $x$ and $y$ real variables used as the edge of an index of summation. The letter $f$ represents our modified Dirichlet character $\chi$ with set of modifications $S$. The symbol $|S|$ means for the cardinality of $S$. We let bold letters like $\bm{\alpha}$ to represent vectors in $\mathbb{R}^d$.

\subsection{Asymptotic notation} We use the standard Vinogradov notation $f(x)\ll g(x)$ or Landau's $f(x)=O(g(x))$ whenever there exists a constant $c>0$ such that $|f(x)|\leq c|g(x)|$, for all $x$ in a set of parameters. When not specified, this set of parameters is an infinite interval $(a,\infty)$ for sufficiently large $a>0$. Sometimes is convenient to indicate the dependence of this constant in other parameters. For this, we use both $\ll_\delta$ or $O_\delta$ to indicate that $c$ may depends on $\delta$. 

The standard $f(x)=o(g(x))$ means that $f(x)/g(x)\to0$ when $x\to a$, where $a$ could be a complex number or $\pm \infty$. 

The notation $\Omega$ was already explained in the first footnote of this paper.

\section{Proof of the main Theorem}\label{section prova}
We will start now the proof of Theorem \ref{Teorema principal} that will be subdivided in several steps. In the rest of the paper, the modified Dirichlet character $f$ will be fixed, and is assumed that $|f(p)|=1$ for all $p\in S$. It is also assumed throughout the paper that the underlying Dirichlet character $\chi$ is primitive.

\subsection{The Dirichlet Series of $f$} We begin by giving a formula for 
$$F(s):=\sum_{n=1}^\infty \frac{f(n)}{n^s}.$$

We recall that (see \cite{aymoneresemblingmobius} for a proof using a Tauberian result) that
$$\sum_{n\leq x}f(n)\ll (\log x)^{|S|},$$
and hence, by the formula
$$F(s)=s\int_{1}^\infty \frac{\sum_{n\leq x}f(n)}{x^{s+1}}dx,$$
$F(s)$ converges in the half plane $\Re(s)>0$, and has $\sigma_c=0$.

Now we easily see that the Euler product representation for $F(s)$ when $\Re(s)>1$ is given by 
\begin{equation}\label{equation Euler product F}
F(s)=\frac{E_\chi(s)}{E_f(s)}L(s,\chi),
\end{equation}
where 
\begin{equation}\label{equation E_f}
E_f(s):=\prod_{p\in S}(1-f(p)p^{-s}).
\end{equation}

Given the analytic properties of $L(s,\chi)$, we have that the representation \eqref{equation Euler product F} holds in all complex plane except on the zeros of $E_f(s)$. As we will see below, this zeros occur only at $\Re(s)=0$ and are all simple.

\subsection{The argument used in the prequel}\label{secao argumento gamma} In the previous paper \cite{aymone_modified}, it was exploited the fact that if for some $N\geq 0$ and function $M(x)$ we have that
\begin{equation}\label{equation argumento gamma}
\int_{1}^\infty\frac{|M(x)|}{x^{1+\sigma}}dx\gg \frac{1}{\sigma^{N+1}}, \,\sigma\to0^+,
\end{equation}
then 
$$M(x)=\Omega((\log x)^N).$$

Indeed, for $N\geq 0$, after a change of variable, we have that
$$\int_1^\infty \frac{(\log x)^N}{x^{1+\sigma}}dx=\frac{\Gamma(N+1)}{\sigma^{N+1}}.$$
Therefore, if \eqref{equation argumento gamma} holds, then $|M(x)|$ cannot be $o((\log x)^N)$.

\subsection{An Harmonic Analysis result} In this paper we exploit the following formula (Theorem 5.4 of \cite{montgomery_livro})
\begin{equation}\label{equation formula de Plancherel}
\int_{1}^\infty\frac{\left|\sum_{n\leq x}f(n)\right|^2}{x^{1+2\sigma}}dx=\frac{1}{2\pi}\int_{-\infty}^\infty \frac{|F(\sigma+it)|^2}{|\sigma+it|^2}dt,
\end{equation}
known as Plancherel's identity and valid for each $\sigma>\max\{\sigma_c,0\}=0$. 

Indeed, in what follows we will show that
$$\int_{-\infty}^\infty \frac{|F(\sigma+it)|^2}{|\sigma+it|^2}dt\gg \frac{1}{\sigma^{|S|-2}},$$
and hence the main Theorem will follow by the above said in Section \ref{secao argumento gamma}.

\section{The underlying Dynamical System} To make the right side of \eqref{equation formula de Plancherel} large we have to make $|F(\sigma+it)|$ large. The only way to achieve this is, by \eqref{equation Euler product F}, to make $E_f(\sigma+it)$ close to $0$.

Indeed, we see that, by denoting $f(p)=e^{2\pi i \theta_p}$ for some $0\leq \theta_p<1$, we have that $E_f(s)^{-1}$ has only simple poles and of the form
$$s_{z,p}=\frac{2\pi i z}{\log p}+\frac{2\pi i \theta_p}{\log p},\,z\in\mathbb{Z},\,p\in S.$$

Let $S=\{p_1,...,p_{|S|}\}$ and
\begin{align*}
\bm{\alpha}&=(\log p_1,...,\log p_{|S|}),\\
\bm{\theta}&=(\theta_1,...,\theta_{|S|}).
\end{align*}

Given a vector $\bm{x}=(x_1,...,x_d)$ in $\mathbb{R}^d$, we define $\bm{x}\mod 1$ to be the vector in the torus $\mathbb{T}^d=[0,1]^d$ defined as
$$\bm{x}\mod 1=(\{x_1\},...,\{x_d\}),$$
where $\{\cdot\}$ stands for the fractional part of a real number.

\begin{lemma}\label{lemma cota inferior pontos bons} Consider the map $\Lambda:\mathbb{T}^{|S|}\to \mathbb{T}^{|S|}$ given by $\Lambda(\bm{x})=\bm{\alpha}+\bm{x}\mod 1$. Let $\sigma>0$ be small and $0\leq r\leq \sigma$. If $e^{1/\sigma}\leq 2\pi n\leq e^{10/\sigma}$ and $\Lambda^n(-\bm{\theta})=n\bm{\alpha}-\bm{\theta}\mod 1$ belongs to the box $[0,10\sigma]^{|S|}$, then as $\sigma\to0^+$ and uniformly in $n$ we have that
$$\frac{|F(\sigma+2\pi i n+ir)|^2}{|\sigma+2\pi i n+ir|^2}\gg \frac{1}{\sigma^{2|S|}}\frac{1}{n^{1+2\sigma}(\log n)^2}.$$
\end{lemma}
\begin{proof} Under the hypothesis that $n\bm{\alpha}-\bm{\theta}\mod 1$ belongs to the box $[0,10\sigma]^{S}$, we have that for each $p\in S$,
$$n\log p-\theta_p=m+O(\sigma),$$
for some integer $m=m(n,p)$. 

Hence, 
$$1-f(p)p^{-\sigma}p^{-2\pi i n}p^{-ir}=1-p^{O(\sigma)}e^{-2\pi i (n \log p -\theta_p)}=1-p^{O(\sigma)}\ll\sigma.$$
Therefore, 
$$\frac{1}{E_f(\sigma+2\pi i n+ir)}\gg \frac{1}{\sigma^{|S|}},\,\sigma\to0^+.$$

On the other hand, for $p\in S$, denoting $\chi(p)\neq 0$ by $e^{2\pi i\lambda_p}$ for some $0\leq \lambda_p<1$, we have that
\begin{align*}
1-\chi(p)p^{-\sigma}p^{-2\pi i n}p^{-ir}&=1-p^{O(\sigma)}e^{-2\pi i (n \log p -\theta_p)+2\pi i (\lambda_p-\theta_p)}\\
&=1-p^{O(\sigma)}e^{2\pi i (\lambda_p-\theta_p)}\\
&=(1-e^{2\pi i (\lambda_p-\theta_p)}+O(\sigma))\\
&\gg 1,
\end{align*}
since $0<|\lambda_p-\theta_p|<1$. In the case that $\chi(p)=0$, the lower bound above is trivial.

Therefore, 
$$E_\chi(\sigma+2\pi i n+ir)\gg 1.$$

Now, by the functional equation for $L(s,\chi)$ ($\chi$ primitive) and classical estimates for the $\Gamma$ function (see Corollary 10.10 of \cite{montgomery_livro}), we have that
$$L(\sigma+2\pi i n+ir,\chi)\gg n^{1/2-\sigma}|L(1-\sigma-2\pi i n-ir,\overline{\chi})|.$$

The Vinogradov-Korobov zero free region for $L$-functions states that they are free of zeros in the region of points $s=\sigma+it$ such that $|t|\geq 1$ and
$$\sigma>1-\frac{c_0}{(\log (10+|t|))^{2/3+\epsilon}},$$
 where $c_0$ is a positive constant and $\epsilon>0$ is small and fixed. Here we refer to \cite{khale} for explicit and more accurate results on this zero free region. Replacing $\sigma$ by $1-\sigma$ above,
we obtain a zero free region for points of the form $s=1-\sigma+it$ such that $0<\sigma\leq \sigma_0$ for some $\sigma_0>0$ and
\begin{equation}\label{equation zero free region}
1\leq |t|\ll e^{c_1/\sigma^{1.4}},
\end{equation}
where $c_1$ is a positive constant.

By doing standard Complex Analysis as in Theorem 11.4 of \cite{montgomery_livro}, inside a region of the type \eqref{equation zero free region} with $c_1/2$ in place of $c_1$, we have that for all $|t|$ sufficiently large,
$$L(s,\chi)\gg \frac{1}{\log |t|},$$
and hence, in our case in which $e^{1/\sigma}\leq 2\pi n\leq e^{10/\sigma}$, we have that for all $\sigma>0$ sufficiently small
$$|L(1-\sigma-2\pi i n-ir,\overline{\chi})|\gg \frac{1}{\log n},$$
and this completes the proof. \end{proof}

\subsection{Linear forms in Logarithms}

We see from above that giving good lower bounds for $F(\sigma+it)$ is related with the Diophantine properties of 
$$\bm{\alpha}=(\log p_1,...,\log p_{|S|}).$$
 
From now on, $\bm{\alpha}$ will denote the vector above. The study of Diophantine properties of logarithms of algebraic numbers have a long history. Here we will need a consequence from Baker's Theory on linear forms in logarithms \cite{baker_I_II_III}. Indeed, if $\bm{m}=(m_0,...,m_{|S|})$ is a vector in $\mathbb{Z}^{|S|+1}$, then, since the coordinates of $\bm{\alpha}$ are linearly independent over $\mathbb{Q}$, we have that
\begin{equation}\label{equation baker}
|m_0+m_1\log p_1+...+m_{|S|}\log p_{|S|}|\geq \frac{c_1}{|\bm{m}|_{|S|+1}^\kappa},
\end{equation}
where $\cdot$ stands for the usual inner product on $\mathbb{R}^d$, $c_1$ and $\kappa$ are positive constants that depend only on the primes in $S$, and
$$|\bm{m}|_{|S|+1}=\max\{|m_j|:0\leq j\leq |S|\}.$$

The inequality \eqref{equation baker} has been extracted from the book \cite{waldschimdt_book}, pg 10, where the author refer to the paper of Fel'dman \cite{feldman_cota_inferior}.
 
\subsection{Discrepancy of uniformly distributed sequences modulo 1} To simplify notation, let $|S|=d$. Another consequence from the fact that the coordinates of $\bm{\alpha}$ are linearly independent over $\mathbb{Q}$ is that $(n\bm{\alpha}-\bm{\theta}\mod 1:n\geq 1)$ is uniformly distributed on the torus $\mathbb{T}^d$.

A natural question that arises in the Theory of uniformly distributed sequences modulo 1 is quantitative results of how it happens this uniform distribution. Given $x\geq 1$, the Discrepancy $D_x^*$ of a sequence $(\bm{x_n})_n$ in the torus $\mathbb{T}^d$ is a way to measure this approximation to the uniform distribution, and is defined by
$$D_x^*:=\sup_J\left|\frac{1}{x}\sum_{n\leq x}\ind_J(\bm{x_n})-\vol(J)\right|,$$
where $\vol$ denotes the Lebesgue measure in $\mathbb{R}^d$, and $J$ runs over all sets of the form 
$$\{\bm{x}\in\mathbb{R}^d:0\leq x_j\leq \beta_j,\,1\leq j\leq d\},$$
for $\bm{\beta}=(\beta_1,...,\beta_d)\in\mathbb{T}^d$.

Note that if $J$ is any set of the form above that could even depend on $x$, we have that 
\begin{equation}\label{equation auxiliar discrepancy}
\sum_{n\leq x}\ind_J(\bm{x_n})=\vol(J)x+O(xD_x^*).
\end{equation}

In general, it can be shown that $xD_x^*=o(x)$ if $(\bm{x_n})_n$ is uniformly distributed modulo $1$, see \cite{kuipers}.

An upper bound for this Discrepancy is given by the following $d$-dimensional Erd\H{o}s-Turán inequality due independently to Koksma \cite{koksma} and to Sz\"usz \cite{szusz} (see also the notes in the chapter on Discrepancy of the book \cite{kuipers}). By denoting 
$$|\bm{m}|_d=:\max\{|m_j|:1\leq j\leq d\},$$ and 
$$r(\bm{m})=\prod_{1\leq j\leq d}\max\{1,|m_j|\},$$ we have that
$$D_x^*\leq C_d\left(\frac{1}{y}+\frac{1}{x}\sum_{0<|\bm{m}|_d\leq y}\frac{1}{r(\bm{m})}\left|\sum_{n\leq x}e^{2\pi i \bm{m}\cdot \bm{x_n}}\right|\right),$$
where $\bm{m}$ runs over points of $\mathbb{Z}^d$ and $\cdot$ is the standard inner product on $\mathbb{R}^d$, and $C_d$ is a positive constant that depends only in $d$.

The following Lemma is a pivotal part of our proof.

\begin{lemma}\label{lemma discrepancy} Let $\bm{x_n}=n\bm{\alpha}-\bm{\theta}\mod 1$. Then the related Discrepancy satisfies
$$D_x^*\ll \frac{1}{x^{\delta}},$$
for some absolute positive $\delta$.
\end{lemma}
\begin{proof}
We have that
$$\left|\sum_{n\leq x}e^{2\pi i \bm{m}\cdot (n\bm{\alpha}-\bm{\theta})}\right|=\left|\sum_{n\leq x}(e^{2\pi i \bm{m}\cdot\bm{\alpha}})^{n}\right|\leq \frac{2}{|e^{2\pi i \bm{m}\cdot\bm{\alpha}}-1|}\ll\frac{1}{\|\bm{m}\cdot\bm{\alpha}\|},$$
where $\|\cdot\|$ stands for the distance to the nearest integer. By \eqref{equation baker}, we have that 
$$\|\bm{m}\cdot\bm{\alpha}\|\gg \frac{1}{|\bm{m}|_d^\kappa}.$$
Combining this with the above $d$-dimensional Erd\H{o}s-Turán inequality, we obtain that
$$D_x^*\ll \frac{1}{y}+\frac{1}{x}\sum_{0<|\bm{m}|_d\leq y} \frac{|\bm{m}|_d^\kappa}{r(\bm{m})}\ll\frac{1}{y}+\frac{y^{\kappa}}{x}(\log y)^d\ll \frac{1}{y}+\frac{y^{\kappa+1}}{x}.$$
Optimizing the last expression above by making $y^{-1}=y^{\kappa+1}x^{-1}$, we obtain that $y=x^{1/(\kappa+2)}$ and hence the Lemma works with
$\delta=1/(\kappa+2)$.
\end{proof}

\subsection{Completion of the proof} We collect now all the arguments above and proceed with:
\begin{proof}[Proof of Theorem \ref{Teorema principal}] We let $Q_T$ to be the set of $n$ such that $n\bm{\alpha}-\bm{\theta}$ modulo $1$ lies in the cube $[0,1/\log T]^{|S|}$, for $T>0$ to be chosen shortly. We let $Q_T(x)$ to be $|\{n\leq x: n\in Q_T\}|$. 

By Lemma \ref{lemma discrepancy} combined with \eqref{equation auxiliar discrepancy}, we have that
$$Q_T(x)=\frac{x}{(\log T)^{|S|}}+O(x^{1-\delta}),$$
where the $O$-term above is uniform on $T$ and $x$, and $\delta>0$ is a constant.

Now, we lower bound:
$$\int_{-\infty}^\infty \frac{|F(\sigma+it)|^2}{|\sigma+it|^2}dt\geq \sum_{n\in Q_T}\int_{2\pi n}^{2\pi n+\sigma}\frac{|F(\sigma+it)|^2}{|\sigma+it|^2}dt.$$

By restricting even more our range to $n\in Q_T$ and $T\leq 2\pi n\leq T^5$, where now we choose $T=e^{1/\sigma}$, we have that Lemma \ref{lemma cota inferior pontos bons} is applicable and hence

$$\int_{-\infty}^\infty \frac{|F(\sigma+it)|^2}{|\sigma+it|^2}dt\gg \frac{\sigma}{\sigma^{2|S|}}\sum_{\substack{T\leq 2\pi n\leq T^5\\n\in Q_T}}\frac{1}{n^{1+2\sigma}(\log n)^2}.$$

Now we do Riemann-Stieltjes integration:

$$\sum_{\substack{T \leq 2\pi n \leq T^5 \\ n\in Q_T}} \frac{1}{n^{1+2\sigma}(\log n)^2}=\int_{T/2\pi}^{T^5/2\pi}\frac{dQ_T(x)}{x^{1+2\sigma}(\log x)^2}.$$

Since 
$$\frac{Q_T(T/2\pi)}{T^{1+2\sigma}(\log T)^2}\ll \sigma^{|S|+2},$$
and 
$$\frac{Q_T(T^5/2\pi)}{(T^5)^{1+2\sigma}(\log T^5)^2}\ll \sigma^{|S|+2},$$
we obtain that
\begin{align*}
\sum_{\substack{T\leq 2\pi n\leq T^5\\n\in Q_T}}\frac{1}{n^{1+2\sigma}(\log n)^2}&=O(\sigma^{|S|+2})
+(1+2\sigma)\int_{T/2\pi}^{T^5/2\pi} \frac{Q_T(x)}{x^{2+2\sigma}}\frac{(1+o(1))}{(\log x)^2}dx\\
&:=O(\sigma^{|S|+2})+J_1.
\end{align*}

Since for $x\geq T/2\pi$ we have $Q_T(x)\gg x/(\log T)^{|S|}$, we obtain that the last integral $J_1$ is lower bounded by
\begin{align*}
J_1&\gg \frac{1}{(\log T)^{|S|}}\int_{T/2\pi}^{T^5/2\pi} \frac{1}{x^{1+2\sigma}(\log x)^2}dx\\
&\geq \frac{1}{(\log T)^{|S|}(\log T^5)^2}\int_{T/2\pi}^{T^5/2\pi} \frac{1}{x^{1+2\sigma}}dx\\
&\gg \sigma^{|S|+2}\left(\frac{1}{\sigma T^{2\sigma}}-\frac{1}{\sigma T^{10\sigma}}\right)\\
&\gg \sigma^{|S|+1}(e^{-2}-e^{-10})\\
&\gg \sigma^{|S|+1}.
\end{align*}

Thus we conclude that
$$\int_{-\infty}^\infty \frac{|F(\sigma+it)|^2}{|\sigma+it|^2}dt\gg \frac{1}{\sigma^{|S|-2}},$$
and hence, by Plancherel's identity that
$$\int_1^\infty \frac{\left|\sum_{n\leq x}f(n)\right|^2}{x^{1+2\sigma}}dx\gg \frac{1}{\sigma^{|S|-2}}.$$
By the argument in Section \ref{secao argumento gamma}, this means that
$$\left|\sum_{n\leq x}f(n)\right|^2=\Omega((\log x)^{|S|-3}),$$
and this completes the proof.
\end{proof}

\noindent\textbf{Acknowledgements.} The research of the first author is funded by FAPEMIG grant Universal APQ-00256-23 and by CNPq grant Universal 403037/2021-2. The third author is funded by CAPES scholarship, process  88887.130630/2025-00.

\bibliographystyle{siam}
\bibliography{ay.bib}

\newpage

{\small{\sc \noindent
Departamento de Matem\'atica, Universidade Federal de Minas Gerais, Av. Ant\^onio Carlos, 6627, CEP 31270-901, Belo Horizonte, MG, Brazil.} \\
\textit{Email address:} \verb|aymone.marco@gmail.com|}

{\small{\sc \noindent
Instituto de Matem\'atica e Estat\'istica, Universidade Federal de Goi\'as, Campus Samambaia, Rua Jacarand\'a, CEP: 74001-970, Goi\^ania - GO, Brazil.} \\
\textit{Email address:} \verb|apchaves@ufg.br|}

{\small{\sc \noindent
Departamento de Matem\'atica, Universidade Federal de Minas Gerais, Av. Ant\^onio Carlos, 6627, CEP 31270-901, Belo Horizonte, MG, Brazil.} \\
\textit{Email address:} \verb|madu-ramos@ufmg.br|}

\end{document}